\documentclass[10pt]{article}
\usepackage{amsmath,amsfonts,bm}
\usepackage{amsthm}
\usepackage{mathrsfs}
\usepackage{amscd}
\usepackage{amssymb}
\usepackage[all]{xy}
\usepackage{graphicx}
\usepackage[svgnames]{xcolor}
\usepackage{tikz}
\usepackage{color}
\usepackage{mathdots}
\usepackage{CJKutf8}

\DeclareSymbolFont{lettersA}{U}{txmia}{m}{it}

\newtheorem{Def}{\bf Definition}
\newtheorem{Thm}{\bf Theorem}
\newtheorem{Lem}{\bf Lemma}
\newtheorem{Prop}{\bf Proposition}
\newtheorem{Rem}{\bf Remark}

\newtheorem{Cor}{\bf Corollary}

\begin{document}
\title{Self-duality of multidimensional continued fractions}
\author{Hiroaki ITO}
\date{}
\maketitle

\begin{abstract}
F.~Schweiger introduced the fibred system in \cite{Schweiger-MCF}, to unify and generalize many known continued fraction algorithms. An advantage of a fibred system is that it often provides a systematic construction of absolutely continuous invariant density. 
In this paper, we define and study the self-duality of fibred systems, a strong symmetry of a given system. 
We show that explicit algebraic self-duality holds in many systems and presents a curious system with "partial" self-duality. 
\end{abstract}

\section{Introduction}
The classical continued fraction is a self map $T$ on $[0,1]$ 
defined by
$$
T:x \longmapsto \frac 1x -\left\lfloor \frac 1x\right\rfloor.
$$
Its absolutely continues invariant probabilistic density is
$$
d\nu= \frac 1{\log 2} \cdot
\frac 1{1+x} dx.
$$
The cylinder set $\Delta[a_1,a_2,\dots,a_n]$ is the interval 
whose elements share initial fraction:
$$
\cfrac {1}{a_1+\cfrac {1}{a_2+\cfrac{1}{\ddots + \cfrac {1}{a_n}}}}.
$$
Then we have
$$
\nu(\Delta[a_1,\dots,a_n])=\nu(\Delta[a_n,\dots,a_1]).
$$
We say that continued fraction algorithm is 
{\bf symmetric in measure} if this equality holds for all cylinder sets.
To see this symmetry, a standard way is to consider its natural extension:
$$
\hat{T}:[0,1]^2\ni (x,y) \mapsto \left(\frac 1x - \left\lfloor \frac 1x \right\rfloor, \frac 1{y+ \left\lfloor \frac 1x \right\rfloor}\right)\in [0,1]^2
$$
with the invariant density
$$
\frac 1{(1+xy)^2}dxdy.
$$
The map $\hat{T}$ is invertible:
$$
\hat{T}^{-1}:
(x,y) \mapsto \left(\frac 1{x + \left\lfloor \frac 1y \right\rfloor}, 
\frac 1y-\left\lfloor \frac 1y \right\rfloor\right)
$$
and the restriction of $\hat{T}^{-1}$
to the second coordinate is equal to $T$. The self-duality immediately follows from this fact.

To make concrete the tractable 
a natural extension for higher dimensional continued fractions,
F.~Schweiger constructed 
the dual algorithm $(B^{\#},T^{\#})$ of the fibred system $(B,T)$.
The pair $(B \times B^{\#}, T \times V^{\#}(k(x)))$ gives 
the natural extension of $(B,T)$ where $V^{\#}(k)$ is a local inverse branch of $T^{\#}$. In this framework, if there exists an isomorphism $\phi$ which satisfies:
\[
  \begin{CD}
     {B^{\#}} @>{T^{\#}}>> {B^{\#}} \\
  @V{\phi}VV    @V{\phi}VV \\
     {B} @>{T}>> {B}
  \end{CD}
\]
then the system is self-dual. 
We say that self-duality is realized by an intertwining map $\phi$. 
We will define an algebraic self-duality. If such a map
$\phi$ is found we simply say that the system $(B, T)$
is {\bf algebraic self-dual}. 
See section \ref{SelfDual} for details. 
In this paper we start with an easy observation:
\begin{Thm} 
\label{self-dual.symmetric}
If the fibred system $(B, T)$ is full and algebraic self-dual, then it is symmetric in measure.
\end{Thm}

However, we do not know when the self-duality holds in general, nor how to construct the intertwining map $\phi$ for a given full-branched fibred system. 
In the later section, we shall construct $\phi$ for several fibred systems in \cite{Schweiger-MCF}
and also give examples of fibred systems
which is not self-dual. 

\section{Invariant measure and self-duality}
\label{SelfDual}

In this chapter, we briefly review the concept of higher dimensional continued fractions by F.~Schweiger and shows Theorem 1. 

We say that the dynamical system $(B, T)$ is a fibred system if $\{B(k): k\in I\}$ is a partition of the set $B$ where $I$ is countable and $T|_{B(k)}$ is injective.

\begin{Def}
The fibred system $(B,T)$ is multidimensional continued fraction $(\bf{ m.c.f.}$, $n$-dimensional c.f.$)$ if \\
$1$. $B\subset \mathbb{R}^n$, \\
$2$. For every digits $k\in I$, there is a matrix $A_{T}(k)=((A_{ij}))\in GL(n+1,\mathbb{Z})$ such that $y=T(x)$, $x\in B(k)$ is given as 
\begin{align*}
y_i=\frac{A_{i0}+\sum_{j=1}^{n}A_{ij}x_j}{A_{00}+\sum_{j=1}^{n}A_{0j}x_j}.
\end{align*}
\end{Def}

\begin{Rem}
For all invertible $(n+1)\times (n+1)$-matrix $(a_{ij})$, we define a transformation $f: \mathbb{R}^{n}\longrightarrow \mathbb{R}^{n}$ satisfies
\begin{align*}
    f(x)_{i}=\frac{a_{i0}+\sum_{j=1}^{n}a_{ij}x_j}{a_{00}+\sum_{j=1}^{n}a_{0j}x_j}, 
\end{align*}
and we denote by $A_{f}$ the matrix $((a_{ij}))$. Then, we can verify $A_{f}A_{g}=A_{f\circ g}$.
\end{Rem}

Since $T|_{B(k)}$ is injective, there exists a local inverse branch of $T$
\begin{align*}
V(k):T(B(k)) \rightarrow B(k), \quad x=V(k)y
\end{align*}
We denote the inverse matrix of $A_{T}(k)$ by $((B_{ij}))$. Then $y=Tx$ is equivalent to 
\begin{align*}
x_i=\frac{B_{i0}+\sum_{j=1}^{n}B_{ij}y_j}{B_{00}+\sum_{j=1}^{n}B_{0j}y_j}.
\end{align*}
where $B_{ij}$ satisfies $B_{00}+\sum_{j=1}^{n}B_{0j}y_j>0$.

\begin{Def}
Let $(B,T)$ be a m.c.f. with matrices $\{A_{T}(k): k\in I\}$. The m.c.f. $(B^{\#},T^{\#})$ is dual algorithm if the following conditions holds$:$\\
$1$. $B(k_1,k_2,\cdots,k_n)\not=\emptyset$ if and only if $B^{\#}(k_n,k_{n-1},\cdots,k_1)\not=\emptyset$, \\
$2$. There is a partition $\{B^{\#}(k), k\in I\}$ of $B^{\#}$ such that the associated matrices $A_{T^{\#}}(k)=((A^{\#}_{ij}))$ of $T^{\#}$ restricted $B^{\#}(k)$ are the {\bf transposed matrices} of $A_{T}(k)$ such that $y=T^{\#}(x)$, $x\in B^{\#}(k)$ is given as 
\begin{align*}
y_i=\frac{A^{\#}_{i0}+\sum_{j=1}^{n}A^{\#}_{ij}x_j}{A^{\#}_{00}+\sum_{j=1}^{n}A^{\#}_{0j}x_j}.
\end{align*}
\end{Def}

Given a multidimensional continued fraction algorithm $(B, T)$, its dual map is formally 
defined by the transpose of $A_T$. We then try to find an appropriate dual space $B^{\#}$ and its decomposition  $\{B^{\#}(k): k\in I\}$ which satisfies condition 1.

After this construction, given an $n$-dimensional continued fraction we set 
\begin{align*}
K(x,y):=\frac{1}{(1+x_1y_1+x_2y_2+\cdots +x_ny_n)^{n+1}},
\end{align*}
and we denote by $\omega(k_1,k_2,\cdots,k_s;y)$ the Jacobian of $V(k_1,k_2,\cdots,k_s)=V(k_1)\circ V(k_2)\circ  \cdots\circ V(k_s)$. 
Then, we can see
\begin{align}\label{2}
K(V(k_1,\dots,k_s)x,y)\omega(k_1,\dots,k_s;x)=K(x,V^{\#}(k_s,\dots,k_1)y)\omega^{\#}(k_s,\dots,k_1;y)
\end{align}
by a straightforward calculation. For any $x\in B$, we define
\begin{align*}
D(x):=\{y\in B^{\#} : x\in \bigcap_{s=1}^{\infty}T^{s}B(k_{s}^{\#}(y),\cdots, k_{1}^{\#}(y))\}.
\end{align*}
Then, it is known that the following assertion holds (see Chapter 3 in \cite{Schweiger-MCF}):
\begin{Prop}\label{invareant-density}
\begin{align*}
h(x)=\int_{D(x)}K(x,y)dy
\end{align*}
is invariant density for $T$.
\end{Prop}

\begin{Def}
A dynamical system $(B,T)$ is called full if $T(\mathring{B(d)})=\mathring{B}$ for all d $\in I$.
\end{Def}

Note that $D(x)=B^{\#}$ if the system $(B,T)$ is full. By Proposition \ref{invareant-density}, we can obtain an invariant measure $\mu^{\#}$ for the dual algorithm $(B^{\#},T^{\#})$.

\begin{Lem}
\label{full-symmetric}
The multidimensional c.f. $(B,T)$ is full, then one has
\begin{align*}
\mu(B(k_1,k_2, \cdots,k_s))=\mu^{\#}(B^{\#}(k_s,k_{s-1},\cdots,k_1)).
\end{align*}
\end{Lem}
\begin{proof}
For all $k_1, \dots. k_s\in I$, by Proposition  \ref{invareant-density},
\begin{align*}
\mu(B(k_1,k_2, \cdots,k_s))
&=\int_{B(k_1,k_2, \cdots,k_s)}\int_{B^{\#}}K(x,y)dydx\\
&=\int_{B}\int_{B^{\#}}K(V(k_1,k_2, \cdots,k_s)x,y)\omega(k_1,k_2, \cdots ,k_s;x)dydx.
\end{align*}
By (\ref{2}), we have
\begin{align*}
\mu(B(k_1,k_2, \cdots,k_s))
&=\int_{B}\int_{B^{\#}}K(x,V^{\#}(k_s,k_{s-1}, \cdots ,k_1)y)\omega^{\#}(k_s,k_{s-1}, \cdots ,k_1;y)dydx\\
&=\int_{B^{\#}}\int_{B}K(V^{\#}(k_s,k_{s-1}, \cdots ,k_1)y,x)\omega^{\#}(k_s,k_{s-1}, \cdots ,k_1;y)dxdy\\
&=\mu^{\#}(B^{\#}(k_s,k_{s-1}, \cdots,k_1)).
\end{align*}
\end{proof}

\begin{Def}
A dynamical system $(B,T)$ is ``algebraic self-dual" on $\mathcal{D}\subset I$ if the diagram
\[\xymatrix{
B^{\#} \ar[r]^{T^{\#}} \ar[d]^{\phi} & B^{\#} \ar[d]\\
B \ar[r]^{T} & B\\
}\]
is commutative and $\phi$ is a bijective, differentiable and measurable function map  such that $\phi(\mathring{B^{\#}(k)})=\mathring{B(k)}$ for all $k\in \mathcal{D}$.
\end{Def}

For the regular continued fraction algorithm $([0,1),T)$, we have
\begin{align*}
    A_T(k)=\left(
\begin{array}{ccc}
0 & 1\\
1 & -k
\end{array} 
\right),  \quad
k=\left\lfloor\frac 1x \right\rfloor.
\end{align*}
Thus it is clearly self-dual since $T=T^{\#}$.

\begin{proof}[Proof of Theorem $\ref{self-dual.symmetric}$]
Note that the map $\phi$ is bijective, since $(B,T)$ is algebraic self-dual (on $\mathcal{D}=I$). By substitution, for all $k_1,k_2,\cdots,k_s\in I$
\begin{align*}
\mu^{\#}(\phi^{-1}B(k_1,k_2,\cdots,k_s))&=\int_{\phi^{-1}B(k_1,k_2,\cdots,k_s)}\int_{B}K(x,y)dydx\\
&=\int_{B(k_1,k_2,\cdots,k_s)}\int_{B^{\#}}K(X,Y)dYdX\\
&=\mu(B(k_1,k_2,\cdots,k_s)).
\end{align*}
Therefore, by Lemma \ref{full-symmetric}, we have
\begin{align*}
\mu(B(k_1,k_2,\cdots,k_s))
&=\mu^{\#}(B^{\#}(k_1,k_2,\cdots,k_s))\\
&=\mu(B(k_s,k_{s-1},\cdots,k_1)).
\end{align*}
\end{proof}

Note that the m.c.f. $(B,T)$ is algebraic self-dual on $\mathcal{D}\subset I$, then for all $k_1,k_2,\cdots,k_s\in \mathcal{D}$
\begin{align*}
    \mu(B(k_1,k_2,\cdots,k_s))=\mu(B(k_s,k_{s-1},\cdots,k_1)).
\end{align*}




The following contents is due to Schweiger (\cite{Schweiger-EToFSaMT}, \cite{Schweiger-IMfMoCFT}).


Let $\mu$ be an invariant measure for the multidimensional continued fraction $(B, T)$. The set function
\begin{align*}
\tau(B(k_1,k_2,\cdots,k_s))=\mu(B(k_{s},k_{s-1},\cdots,k_1))
\end{align*}
is called the polar measure for $(B, T)$.

\begin{Thm}
The kernel measure $\kappa$ is absolutely continuous with respect to Lebesgue measure $\lambda$ if and only if the measure $\tau$ coincides with the invariant measure $\mu$.
\end{Thm}


\section{Garrity-Schweiger Algorithm}
Let $E^{n+1}:=\{x\in\mathbb{R}^{n+1}_{>} : x_0\ge x_1\ge\cdots \ge x_{n}\ge 0\}$. We define
\begin{align*}
G(x)=(x_{1},\dots, x_{n}, x_0-x_1-kx_{n})\in E^{n+1}, \quad 
    k=\left\lfloor\frac{x_0-x_1}{x_n}\right\rfloor.
\end{align*}
With the help of the projection
\begin{align*}
    p:E^{n+1}\longrightarrow \Delta, \quad p(x_0,x_1,\dots,x_n)=\left(\frac{x_1}{x_0},\dots, \frac{x_n}{x_0}\right),
\end{align*}
we obtain the bottom map $T: \Delta \longrightarrow \Delta$ which makes the diagram
\[
  \begin{CD}
     {E^{n+1}} @>{G}>> {E^{n+1}} \\
  @V{p}VV    @V{p}VV \\
     {\Delta} @>{T=T_G}>> {\Delta}
  \end{CD}
\]
commutative. The map $T$ is 
\begin{align*}
T(x)=\left(\frac{x_2}{x_1}, \frac{x_3}{x_1},\cdots\frac{x_n}{x_1}, \frac{1-x_1-kx_n}{x_1}\right),\quad  \displaystyle k=k(x)=\left\lfloor\frac{1-x_1}{x_n}\right\rfloor.
\end{align*}
The $1$-time partition of $\Delta$ is 
\begin{align*}
\Delta(k)=\{x\in B : 1-x_1-kx_n\ge 0> 1-x_1-(k+1)x_n\},\quad k\in\{0,1,2,\dots\}
\end{align*}
and this fibred system is full.

\begin{figure}[h]
\begin{center}
\scalebox{1}[1]{
\begin{tikzpicture}[domain=-1.2:5.5]
\draw[->] (-0.2,0) -- (5.2,0) ;
\draw[->] (0,-0.2) -- (0,2.2) ;
\draw [<-, domain=-1.1:0.1]plot(\x,4*\x);
\draw [dashed] (0,0) -- (4,-4) ;
\draw  (-1,-4) -- (4,-4) ;
\draw  (4,1) -- (4,-4) ;
\draw  (4,1) -- (0,0) ;
\draw  (4,1) -- (-1,-4) ;
\draw  (-1,-4) -- (2,1/2);
\draw  (-1,-4) -- (4/3,1/3);
\draw  (-1,-4) -- (1,1/4);
\draw [dashed] (4,-4) -- (2,1/2);
\draw [dashed] (4,-4) -- (4/3,1/3);
\draw [dashed] (4,-4) -- (1,1/4);
\fill (3,-0.5) node[below] {\small{0}};
\fill (1.5,-0.5) node[below] {\small{1}};
\fill (1,-0.5) node[below] {\small{2}};
\end{tikzpicture}}
\end{center}
\caption{The $1$-time partition of $(\Delta, T_{G})$ for $n=3$.}
\end{figure}
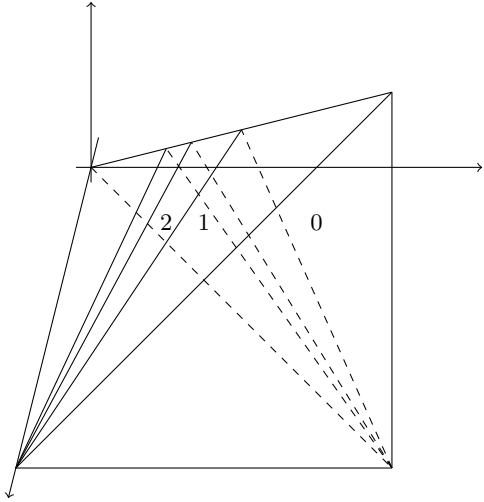

This continued fraction algorithm for $n=2$ was introduced by Garrity in \cite{Garrity-Opsfan}. The $n$-dimensional Garrity's map is also studied in \cite{BST-OtsLeosmcfa}. On the other hand, the following map was introduced by Schweiger in \cite{Schweiger-Aneojtaweim}:
\begin{align*}
    F:x\longmapsto \left(\frac{x_2}{x_1}-\left\lfloor\frac{x_2}{x_1}\right\rfloor, \frac{x_3}{x_1},\dots, \frac{x_n}{x_1}, \frac{1}{x_1}-1\right)
\end{align*}
on $[0,1)\times \mathbb{R}_{\ge}^{n-1}$. The dynamical system $([0,1)\times \mathbb{R}_{\ge}^{n-1}, F)$ is isomorphic to the dual algorithm of $(\Delta, T_G)$. It is known that the m.c.f. $(\Delta, T_G)$ is ergodic with respect to the Lebesgue measure for $n=2$ (see \cite{Messaoudi.Nogueira.Schweiger-Epotp}).
However, this algorithm is not topological convergent and it is not known if $F$ is ergodic for $n\ge 3$ (see \cite{Schweiger-Aneojtaweim}). In this paper, we call the m.c.f. $(\Delta, T_G)$ Garrity-Schweiger algorithm. 

We show that the m.c.f. $(\Delta, T_G)$ is algebraic self-dual. Since
\begin{align*}
A_{T}(k)=
\left(
\begin{array}{ccccc}
0 & 1 & 0 & \cdots & 0\\
0 & 0 & 1 & ~ & \vdots \\
\vdots & ~ & ~ & \ddots & 0 \\
0 & ~  &  ~ & ~ & 1\\
1 & -1 & 0  & \cdots & -k
\end{array} 
\right) \quad \textbf{on } \Delta(k), 
\end{align*}
the dual map $T^{\#}$ is 
\begin{align*}
T^{\#}(x)=\left(\frac{1-x_n}{x_n},\frac{x_1}{x_n}, \cdots ,\frac{x_{n-2}}{x_n}, \frac{x_{n-1}-kx_n}{x_n}\right).
\end{align*}
This dual map is defined on
\begin{align*}
\Delta^{\#}=\{x\in \mathbb{R}^{n} : x_{i}\ge 0, 1\le i \le n-1, 0\le x_n<1\}=\mathbb{R}^{n-1}_{\ge}\times [0,1)
\end{align*}
and the $1$-time partition is given by 
\begin{align*}
\Delta^{\#}(k)=\{x\in \Delta^{\#} :  x_{n-1}-kx_{n}\ge 0 >x_{n-1}-(k+1)x_{n}\}, \quad k=k^{\#}(x)=\left\lfloor\frac{x_{n-1}}{x_n}\right\rfloor.
\end{align*}

Since the Garrity-Schweiger algorithm is full system, the invariant density $h$ is
\begin{align*}
\int_{\mathbb{R}^{n-1}_{\ge}\times [0,1)}K(x,y)dy\sim\frac{1}{x_1\cdots x_{n-1}(1+x_n)}.
\end{align*}

\begin{figure}[h]
\begin{center}
\begin{tabular}{cc}
\scalebox{1}[1]{
\begin{tikzpicture}[domain=-0.2:5.5]
\draw[very thin,color=gray, dashed] (0,5) -- (5,5);
\draw[thin] (5,0) -- (5,5);
\draw[->] (-0.2,0) -- (5.2,0) ;
\draw[->] (0,-0.2) -- (0,5.2) ;
\draw [domain=0:5]plot(\x,\x);
\draw [domain=2.5:5]plot(\x,{-\x+5)});
\draw [domain=5/3:5]plot(\x,{-\x/2+5/2}) ;
\draw [domain=5/4:5]plot(\x,{-\x/3+5/3}) ;
\fill (3.7,2.6) node[below] {\small{0}};
\fill (2.5,1.9) node[below] {\small{1}};
\fill (1.7,1.6) node[below] {\small{2}};
\end{tikzpicture}}
\scalebox{1}[1]{
\begin{tikzpicture}[domain=-0.2:5.5]
\draw[very thin,color=gray, dashed] (0,5) -- (5,5);
\draw[thin] (5,0) -- (5,5);
\draw[->] (-0.2,0) -- (5.2,0) ;
\draw[->] (0,-0.2) -- (0,5.2) ;
\draw [domain=0:5]plot(\x,\x);
\draw [domain=2.5:5]plot(\x,{-\x+5)});
\draw [domain=5/3:5]plot(\x,{-\x/2+5/2}) ;
\draw [domain=5/4:5]plot(\x,{-\x/3+5/3}) ;
\draw[thin] (10/3,5/3) -- (5,5);
\draw[thin] (15/4,5/4) -- (5,5);
\draw[thin] (5/2,5/4) -- (5/2,5/2);
\draw[thin] (3,1) -- (5/2,5/2);
\draw[thin] (2,1) -- (5/3,5/3);
\draw[thin] (5/2,5/6) -- (5/3,5/3);
\fill (3.1,2.7) node[below] {\small{00}};
\fill (3.7,2.1) node[below] {\small{01}};
\fill (2.1,2) node[below] {\small{10}};
\fill (2.7,1.6) node[below] {\small{11}};
\fill (1.6,1.5) node[below] {\small{20}};
\end{tikzpicture}}
\end{tabular}
\end{center}
\caption{The $1$-time and $2$-time partition of $(\Delta, T_{G})$.}
\end{figure}
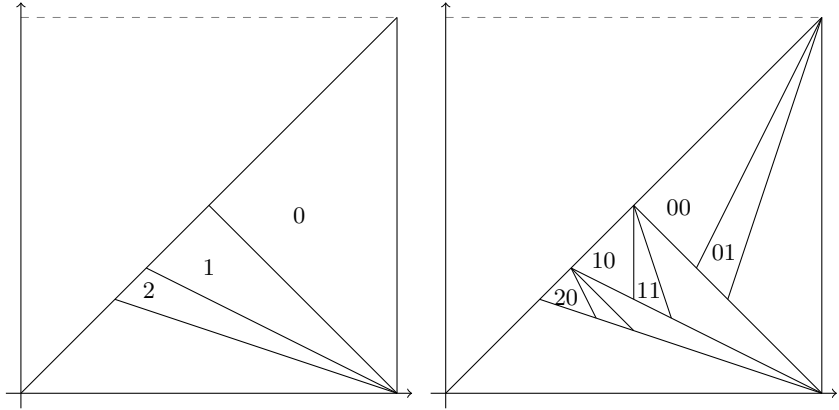

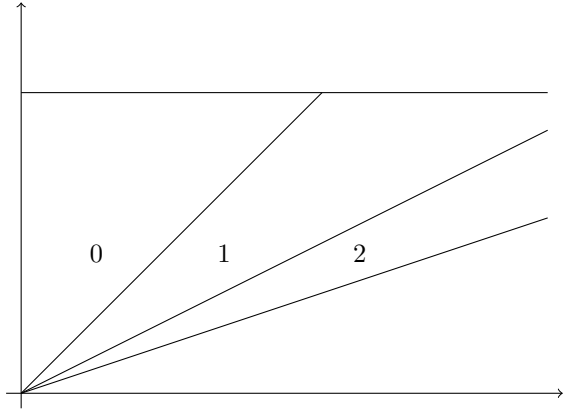
\begin{figure}[h]
\begin{center}
\begin{tabular}{ccc}
\scalebox{1}[1]{
\begin{tikzpicture}[domain=-0.2:7.5]
\draw[->] (-0.2,0) -- (7.2,0) ;
\draw[->] (0,-0.2) -- (0,5.2) ;
\draw [domain=0:4]plot(\x,\x);
\draw [domain=0:7]plot(\x,4);
\draw [domain=0:7]plot(\x,\x/2) ;
\draw [domain=0:7]plot(\x,\x/3) ;
\fill (1,2.1) node[below] {0};
\fill (2.7,2.1) node[below] {1};
\fill (4.5,2.1) node[below] {2};
\end{tikzpicture}}
\end{tabular}
\end{center}
\caption{The $1$-time partition of $(\Delta^{\#}, T_{G})$.}
\end{figure}

We found that this algorithm is self-dual for $n=2$ and 
the matrix $A_{\phi}$ is given by
\begin{align*}
A_{\phi}=
\left(
\begin{array}{ccc}
1 & 1 & 0\\
1 & 0 & 0\\
0 & 0 & 1\\
\end{array} 
\right)
\end{align*}
Here we describe our heuristic method to find such a matrix.
First we assume that $A_\phi=((a_{ij}))$ has integer entries. From 
$A_\phi A_{T^{\#}}= A_{T}A_\phi$, we see $A_\phi$ is symmetric.
Assume that $\phi$ sends $\Delta^{\#}(k) \cap \Delta^{\#}(k+1)$ to $\Delta(k) \cap\Delta(k+1)$.
In particular, if 
$\phi(0,0)=(1, 0)$, then we see 
\begin{align*}
a_{11}=a_{21}, \quad a_{31}=0.
\end{align*}
Put $a_{11}=x$, $A_{\phi}$ has to have a form
\begin{align*}
\left(
\begin{array}{ccc}
x & x & 0\\
x & * & *\\
0 & * & *\\
\end{array} 
\right).
\end{align*}
Further if $\phi(k,1)=(\frac{1}{k+1}, \frac{1}{k+1})$, then 
\begin{align*}
    \frac{x+ka_{22}}{x+kx}=
    \frac{ka_{32}+a_{33}}{x+kx}=\frac{1}{k+1}.
\end{align*}
Therefore we have $x=1$, $a_{22}=a_{32}=0, a_{33}=1$ and the condition $A_\phi A_{T^{\#}}(k)= A_{T}(k)A_\phi$, $\phi(\Delta^{\#}(k))=\Delta(k)$ 
are guaranteed. Thus we obtain the following.

\begin{Prop}
The Garrity-Schweiger algorithm is algebraic self-dual. And
\begin{align*}
A_{\phi}=\left(
\begin{array}{cccc}
1 & \cdots & 1 & ~ \\
\vdots & \iddots & ~ & ~ \\
1 & ~ & ~ & ~ \\
~ & ~ & ~ &  1\\
\end{array} 
\right).
\end{align*}
\end{Prop}

\begin{proof}
By a straightforward calculation, we can see that for all $x\in\Delta$
\begin{align*}
\phi \circ T^{\#}(x)=T\circ \phi(x).
\end{align*}

Let $\phi(B_1,B_2,\dots,B_n)=(b_1,b_2,\dots,b_n)$. Then, since
\begin{align*}
A_{\phi}^{-1}=
\left(
\begin{array}{ccccc}
 & & & 1 & \\
 & & 1 & -1 &\\
 & \iddots & \iddots & &\\
1 & -1 & & &\\
 & & & & 1
\end{array} 
\right),
\end{align*}
we have 
\begin{align*}
    B_{n-1}-\left\lfloor\frac{B_{n-1}}{B_n}\right\rfloor B_n
    =\frac{\displaystyle1-b_{1}-\left\lfloor\frac{1-b_{1}}{b_n}\right\rfloor b_n}{b_{n-1}}.
\end{align*}
Therefore, we can see $\phi(\mathring{\Delta^{\#}(k)})=\mathring{\Delta(k)}$ for all $k\in \mathbb{Z}_{\ge 0}$.
\end{proof}

\begin{Cor}
The Garrity-Schweiger algorithm is symmetric in measure.
\end{Cor}

\section{Selmer Algorithm}
Let $E^{n+1}:=\{x\in\mathbb{R}^{n+1}_{>} : x_0\ge x_1\ge\cdots \ge x_{n}\ge 0\}$. Then define
\begin{align*}
    x\in E^{n+1}\longmapsto x^{'}=(x_0-x_n, x_2,\dots, x_{n}).
\end{align*}
There is an index $i=i(x)$, $0\le i\le n$ such that
\begin{align*}
S(x)=(x_{1},\dots, x_{i}, x_0-x_n, \dots, x_{n})\in E^{n+1}.
\end{align*}

We obtain the bottom map $T_S: \Delta \longrightarrow \Delta$ which makes the diagram
\[
  \begin{CD}
     {E^{n+1}} @>{S}>> {E^{n+1}} \\
  @V{p}VV    @V{}VV \\
     {\Delta} @>{T=T_{S}}>> {\Delta}
  \end{CD}
\]
commutative. Since
\begin{align*}
A_{T}(i)=
    \bordermatrix{
  & & & & i & & & n\cr
0 & & 1 & & & & &\cr
  & &   & \ddots & & & &\cr
  & &   & & 1 & & &\cr
i & 1 & &  & & & & -1\cr
  & & & & & 1 & &\cr
  & & & & & & \ddots & \cr
  & & & & & & & 1 \cr
} \quad \textbf{on}~ \Delta(i), 
\end{align*}
the $1$-time partition is
\begin{align*}
    \Delta(i)=\{x\in\Delta : x_{i}> 1-x_n \ge x_{i+1}\}
\end{align*}
where $x_0=1$, $x_{n+1}=0$. We can see
\begin{align*}
    T(\Delta(i))=\{x\in\Delta : x_{i}+x_n\ge 1\}.
\end{align*}
Therefore, for $i=0,1,\dots, n-1$
\begin{align*}
    T(\Delta(i))=\bigcup_{i\le j}\Delta(j), 
\end{align*}
and
\begin{align*}
    T(\Delta(n))=\Delta(n-1)\cup\Delta(n).
\end{align*}

The m.c.f. $(\Delta, T)$ is not full, but $(X=\Delta(n-1)\cup\Delta(n), T)$ is full-branched system.
The dual map of Selmer's algorithm $(X, T)$ is defined on $X^{\#}=\mathbb{R}^{n}_{\ge}$. It is known that Selmer algorithm $(X, T)$ is ergodic and admits an absolutely continuous invariant measure (see \cite{Schweiger-MCF}). 

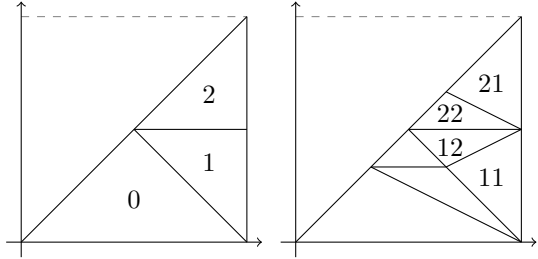
\begin{figure}[h]
\begin{center}
\begin{tabular}{ccc}
\scalebox{1}[1]{
\begin{tikzpicture}[domain=-0.2:4]
\draw[very thin,color=gray, dashed] (0,3) -- (3,3);
\draw[thin] (3,0) -- (3,3);
\draw[->] (-0.2,0) -- (3.2,0) ;
\draw[->] (0,-0.2) -- (0,3.2) ;
\draw [domain=0:3]plot(\x,\x);
\draw [domain=1.5:3]plot(\x,{-\x+3)});
\draw [domain=1.5:3]plot(\x,3/2) ;
\fill (2.5,2.2) node[below] {2};
\fill (2.5,1.3) node[below] {1};
\fill (1.5,0.8) node[below] {0};
\end{tikzpicture}}
\scalebox{1}[1]{
\begin{tikzpicture}[domain=-0.2:4]
\draw[very thin,color=gray, dashed] (0,3) -- (3,3);
\draw[thin] (3,0) -- (3,3);
\draw[->] (-0.2,0) -- (3.2,0) ;
\draw[->] (0,-0.2) -- (0,3.2) ;
\draw [domain=0:3]plot(\x,\x);
\draw [domain=1.5:3]plot(\x,{-\x+3)});
\draw [domain=1.5:3]plot(\x,3/2) ;
\draw [domain=2:3]plot(\x,\x/2);
\draw [domain=2:3]plot(\x,-\x/2+3) ;
\draw [domain=1:2]plot(\x,1) ;
\draw [domain=1:3]plot(\x,{-\x/2+3/2)});
\fill (2.6,2.35) node[below] {21};
\fill (2.05,1.5) node[below] {12};
\fill (2.6,1.1) node[below] {11};
\fill (2.05,1.95) node[below] {22};
\end{tikzpicture}}
\end{tabular}
\end{center}
\caption{The $1$-time and $2$-time partition of $(\Delta, T_{S})$.}
\end{figure}

\begin{figure}[h]
\begin{center}
\begin{tabular}{ccc}
\scalebox{1}[1]{
\begin{tikzpicture}[domain=-0.2:4]
\draw[->] (-0.2,0) -- (3.2,0) ;
\draw[->] (0,-0.2) -- (0,3.2) ;
\draw [domain=0:3]plot(\x,\x);
\fill (1,2.2) node[below] {1};
\fill (2,1.3) node[below] {2};
\end{tikzpicture}}
\end{tabular}
\end{center}
\caption{The $1$-time partition of $(X^{\#}, T^{\#}_{S})$.}
\end{figure}
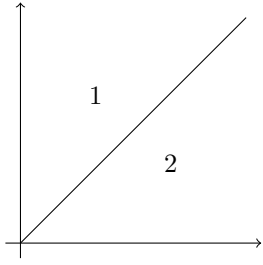

Note that the $1$-time partition of $X^{\#}$ is 
\begin{align*}
    X^{\#}(n-1)&=\{x\in X^{\#} : x_{n-1}\le x_n\},\\
    X^{\#}(n)&=\{x\in X^{\#} : x_{n-1}\ge x_n\}.
\end{align*}
We construct the intertwining map $\phi$ for the Selmer algorithm. 

For $n=2$, the same method in the previous chapter gives
\begin{align*}
    A_{\phi}=
\left(
\begin{array}{ccc}
2 & 1 & 1\\
1 & 1 & 1\\
1 & 1 & 0
\end{array} 
\right).
\end{align*}
A simple analogy for the $n$-dimensional case works fine and we obtain

\begin{Prop}
Selmer's algoriyhm is algebraic self-dual. And 
\begin{align*}
A_{\phi}=
\left(
\begin{array}{ccccc}
2 & \cdots & 2& 1 & 1\\
\vdots & \iddots & ~ & \vdots & \vdots \\
2 & ~ & ~ & ~ & ~ \\
1 & \cdots  & ~ & ~ & 1\\
1 & \cdots & ~  & 1 & 0
\end{array} 
\right).
\end{align*}
\end{Prop}

\begin{proof}
We can see that for all $k\in\{n-1, n\}$
\begin{align*}
    A_{\phi}A_{T^{\#}}(k)=A_{T}(k)A_{\phi}.
\end{align*}

Let $\phi(B_{1},\dots,B_n)=(b_{1},\dots,b_n)$. Then, for all $(B_{1},\dots,B_n)\in X^{\#}(n-1)$, 
\begin{align*}
b_{n-1}+b_{n}-1=\frac{B_{n-1}}{2+2B_1+\cdots+2B_{n-2}+B_{n-1}+B_{n}}>0
\end{align*}
and
\begin{align*}
1-2b_n=\frac{B_{n}-B_{n-1}}{2+2B_1+\cdots+2B_{n-2}+B_{n-1}+B_{n}}>0.
\end{align*}
Thus, $\phi(X^{\#}(n-1))\subset \Delta(n-1)$. We show $\phi(X^{\#}(n-1))\supset \Delta(n-1)$. Since 
\begin{align*}
A_{\phi}^{-1}=
\left(
\begin{array}{ccccc}
 & & 1& -1 & \\
 & \iddots & -1 & &\\
1 & \iddots & ~ & ~ & ~ \\
-1 & & & 1 & 1\\
 & & & 1 & -1
\end{array} 
\right),
\end{align*}
for all $(b_{1},\dots,b_n)\in\Delta(n-1)$, we can verify that $B_i\ge 0$, $i=1,\dots,n$ and 
\begin{align*}
    B_{n}-B_{n-1}=\frac{1-2b_n}{b_{n-2}-b_{n-1}}>0.
\end{align*}
Thus, $\phi(X^{\#}(n-1))= \Delta(n-1)$. Simiraly, we can see
$\phi(X^{\#}(n))= \Delta(n)$.
\end{proof}

\begin{Cor}
Selmer's algorithm $(X,T)$ is symmetric in measure. i.e., for all $a_1,a_2,\cdots,a_s\in\{n-1,n\}$
\begin{align*}
    \mu(\Delta(a_1,a_2,\cdots,a_s))=\mu(\Delta(a_s,a_{s-1},\cdots,a_1)).
\end{align*}
\end{Cor}

\section{Brun Algorithm}
Let $E^{n+1}:=\{x\in\mathbb{R}^{n+1}_{>} : x_0\ge x_1\ge\cdots \ge x_{n}\ge 0\}$. Define
\begin{align*}
    x\in E^{n+1}\longmapsto x^{'}=(x_0-x_1, x_2,\dots, x_{n}).
\end{align*}
Then, there is an index $i=i(x)$, $0\le i\le n$ such that
\begin{align*}
B(x)=(x_{1},\dots, x_{i}, x_0-x_1, \dots, x_{n})\in E^{n+1}.
\end{align*}
We obtain the bottom map $T: \Delta \longrightarrow \Delta$ which makes the diagram
\[
  \begin{CD}
     {E^{n+1}} @>{B}>> {E^{n+1}} \\
  @V{p}VV    @V{p}VV \\
     {\Delta} @>{T=T_B}>> {\Delta}
  \end{CD}
\]
commutative, where $p(x)=\left(\frac{x_1}{x_0},\frac{x_2}{x_0},\dots, \frac{x_n}{x_0}\right)$. The $1$-time partition of $\Delta$ is 
\begin{align*}
    \Delta(i)&=\{x\in\Delta : x_i\ge 1-x_1> x_{i+1}\}
\end{align*}
where $x_0=1$, $x_{n+1}=0$. For digits $i\in\{0,1,2,\dots\}$,  the matrix $A_{T}(i)$ is 
\begin{align*}
A_{T}(i)=
    \bordermatrix{
  & & & & i & & & \cr
  & & 1 & & & & &\cr
  & &   & \ddots & & & &\cr
  & &   & & 1 & & &\cr
i & 1 & -1 &  & & & &\cr
  & & & & & 1 & &\cr
  & & & & & & \ddots & \cr
  & & & & & & & 1 \cr
} \quad \textbf{on}~ \Delta(i).
\end{align*}

The dual space of Brun's algorithm $(\Delta, T)$ is $\Delta^{\#}=\mathbb{R}_{\ge}\times [0,1)^{n-1}$ and the $1$-time partition of $\Delta^{\#}$ is 
\begin{align*}
    \Delta^{\#}(0)&=\{x\in\Delta^{\#} : x_1\ge1\},\\
    \Delta^{\#}(i)&=\{x\in\Delta^{\#} : 0\le x_1<1, x_j<x_i, 1\le j\le n\}.
\end{align*}

We found the intertwining map $\phi$ of Brun algorithm for only $n=2$ case:
\begin{align*}
A_{\phi}=
\left(
\begin{array}{ccc}
1 & 1 & 0\\
1 & 0 & 0\\
0 & 0 & 1\\
\end{array} 
\right).
\end{align*}

And, for $n\ge 3$, we confirm that intertwining map $\phi$
\begin{align*}
\phi(x)=\left(\frac{1}{1+x_1},\frac{x_{2}}{1+x_1},\frac{x_{2}x_{3}}{1+x_1}, \cdots,\frac{x_2\cdots x_{n}}{1+x_1}\right)
\end{align*}
works in $\Delta(0)$.

Now, let $(\Delta, T_M)$ be a multidimensional continued fraction with matrices
\begin{align*}
A_{T_M}(i)=
    \bordermatrix{
  & & & & i & & & \cr
  & & 1 & & & & &\cr
  & &   & \ddots & & & &\cr
  & &   & & 1 & & &\cr
i & 1 & -N &  & & & &\cr
  & & & & & 1 & &\cr
  & & & & & & \ddots & \cr
  & & & & & & & 1 \cr
} \quad \textbf{on}~ \Delta(i, N).
\end{align*}

This algorithm is called Brun multiplicative algorithm. Since $\Delta^{\#}=[0,1]^n$, we can get the invariant density $\mu$ for $T_M$. 
It is known that the multiplicative version is ergodic (see \cite{Schweiger-MCF}). However, by using Mathematica, we observed
\begin{align*}
\mu(\Delta[(1,1),(2,1)])\not=\mu(\Delta[(1,1),(2,1)])
\end{align*}
for dimension $n=2$.

\section{Poincar\'{e} Algorithm}
\begin{CJK*}{UTF8}{goth}
Finally, we give an example that is not self-dual.
Note that this algorithm below is conjugate to the original
Poincar\'{e} algorithm. There are maps $F, G$ that the original map
is $F\circ G$, but ours is $G\circ F$, where $G$ is the sorting map
into non-increasing order. See Chapter 21 of \cite{Schweiger-MCF}.
\end{CJK*}

Let $E^{n+1}:=\{x\in\mathbb{R}^{n+1}_{>} : x_1\ge x_2\ge\cdots \ge x_{n+1}\ge 0\}$. Then define
\begin{align*}
    x\in E^{n+1}\longmapsto x^{'}=(x_1-x_2, x_2-x_3,\dots, x_{n}-x_{n+1},x_{n+1}).
\end{align*}
There is an element $\sigma$ of symmetric group $\mathcal{S}_{n+1}$ such that
\begin{align*}
P(x)=(x'_{\sigma(1)}, x'_{\sigma(2)}, \dots, x'_{\sigma(n+1)})\in E^{n+1}.
\end{align*}

In this section, we consider the normalized map $T_P: \Delta \longrightarrow \Delta$ which makes the diagram
\[
  \begin{CD}
     {E^{n+1}} @>{P}>> {E^{n+1}} \\
  @V{p}VV    @V{p}VV \\
     {\Delta} @>{T=T_P}>> {\Delta}
  \end{CD}
\]
commutative, where $p(x)=\left(\frac{x_2}{x_1},\frac{x_3}{x_1},\dots, \frac{x_{n+1}}{x_1}\right)$.

Then for all digit $\sigma\in \mathcal{S}_{n+1}$, $y=Tx$, $x\in\Delta(\sigma)$ is given by

\begin{align*}
y_{i}=\frac{\sum_{j=1}^{n+1}A_{\sigma^{-1}(i)j}x_{j-1}}{\sum_{j=1}^{n+1}A_{\sigma^{-1}(1)j}x_{j-1}}
\end{align*}
where $x_0=1$ and 
\begin{align*}
A_{T}(e)=((A_{i,j})):=
\left(
\begin{array}{ccccc}
1 & -1&  &  \\
  & \ddots& \ddots & \\
  & ~ & 1 &-1\\
  & ~ & ~ & 1\\
\end{array} 
\right).
\end{align*}

The dual map defined on $\Delta^{\#}=\mathbb{R}^{n}_{>}$. Therefore the invariant density is 
\begin{align*}
\frac{1}{x_1x_2\cdots x_n}.
\end{align*}
The invariant measure $\mu$ is infinite. Note that
\begin{align*}
A_{T}(\sigma)=((A_{\sigma^{-1}(i)j}))=
\bordermatrix{
  & & \sigma^{-1}(i) & \sigma^{-1}(i)+1 & & n \cr
  & & \vdots & \vdots & & \vdots\cr
i & \cdots & 1 & -1 & & \cr
  & & & & & \cr
j & \cdots & & & & 1 \cr
  & & & & & \cr
},
\end{align*}
\begin{align*}
A_{T^{\#}}(\sigma)=
\bordermatrix{
  & & \sigma(1) & & \sigma(i) & & \sigma(i-1) &  \cr
 1& \cdots & 1 & & \vdots & & \vdots & \cr
  & & & & & & & \cr
 i& \cdots & & & 1 &\cdots & -1 & \cr
  & & & & & & &
}.
\end{align*}
Then, we have 
\begin{align*}
\Delta(\sigma)&=\left\{
\begin{array}{ll}
\{x\in \Delta : &x_{\sigma^{-1}(i)-1}-x_{\sigma^{-1}(i)}>x_{\sigma^{-1}(i+1)-1}-x_{\sigma^{-1}(i+1)}, ~\text{for } i=1,2,\dots, n ~(i\not=j-1,j), \\
&x_{\sigma^{-1}(j-1)-1}-x_{\sigma^{-1}(j-1)}>x_{n},~
x_{n}>x_{\sigma^{-1}(j+1)-1}-x_{\sigma^{-1}(j+1)}\}, ~\text{if } j\not=1,n+1,\\
[10pt]
\{x\in \Delta : &x_{n}>x_{\sigma^{-1}(2)-1}-x_{\sigma^{-1}(2)}, \\
&x_{\sigma^{-1}(i)-1}-x_{\sigma^{-1}(i)}>x_{\sigma^{-1}(i+1)-1}-x_{\sigma^{-1}(i+1)}, ~\text{for } i=2,\dots, n\},~\text{if } j=1,\\
[10pt]
\{x\in \Delta : &x_{\sigma^{-1}(i)-1}-x_{\sigma^{-1}(i)}>x_{\sigma^{-1}(i+1)-1}-x_{\sigma^{-1}(i+1)}, ~\text{for } i=1,2,\dots, n-1, \\
&x_{\sigma^{-1}(n)-1}-x_{\sigma^{-1}(n)}>x_{n}\}, ~\text{if } j=n+1
\end{array} \right. 
\end{align*}
and
\begin{align*}
\Delta^{\#}(\sigma)&=\{x\in \Delta^{\#} : x_{\sigma(i+1)-1}-x_{\sigma(i)-1}>0 ~\text{for } i=1,\dots, n\}
\end{align*}
where $x_0=1$ and $j$ is a integer satisfies $\sigma(n+1)=j$.

\begin{figure}[h]
\begin{center}
\begin{tabular}{cc}
\scalebox{1}[1]{
\begin{tikzpicture}[domain=-0.2:5.5]
\draw[very thin,color=gray, dashed] (0,5) -- (5,5);
\draw[thin] (5,0) -- (5,5);
\draw[->] (-0.2,0) -- (5.2,0) ;
\draw[->] (0,-0.2) -- (0,5.2) ;
\draw [domain=0:5]plot(\x,\x);
\draw [domain=5/2:5]plot(\x,{-\x+5)});
\draw [domain=0:5]plot(\x,{\x/2}) ;
\draw [domain=5/2:5]plot(\x,{2*\x-5});
\draw [domain=5/2:15/4]plot(\x,5/2);
\draw [domain=5/3:10/3]plot(10/3,{\x});
\draw [domain=3:5]plot(\x,{3*\x/2-5/2});
\draw [domain=5/2:15/4]plot(\x,\x-5/2);
\draw [domain=0:5/3]plot(10/3,{\x});
\draw [domain=3:5]plot(\x,{-\x/2+5/2});
\fill[lightgray] (10/3,10/3)--(5/2,5/2)--(10/3,5/2)--cycle; 
\fill[lightgray] (10/3,5/3)--(10/3,5/6)--(15/4,5/4)--cycle; 
\fill (2.1,0.8) node[below] {e};
\fill (3.5,0.9) node[below] {(12)};
\fill (2.3,1.9) node[below] {(23)};
\fill (3.3,2.9) node[below] {(123)};
\fill (4.3,3) node[below] {(13)};
\fill (4.3,1.8) node[below] {(132)};
\end{tikzpicture}}
\end{tabular}
\end{center}
\caption{$\Delta[(12), (123)]$ and $\Delta[(123), (12)]$}
\end{figure}
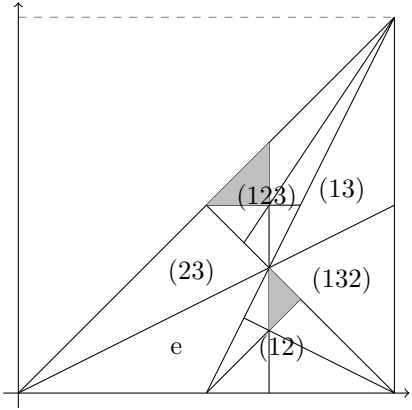

For $n=2$, by direct computation we have
\begin{align*}
    \mu(\Delta[(12),(123)])\not=\mu(\Delta[(123),(12)]).
\end{align*}


Therefore from Theorem \ref{self-dual.symmetric}, this algorithm is not self-dual for $n=2$.

All the same, we shall prove that this algorithm is self-dual on  $\Delta[e]\cup\Delta[(13)]\cup\Delta[(123)]\cup\Delta[(132)]$
with the intertwining map $\phi$.

Let us explain our empirical method to find this intertwining map.
At first, we follow the heuristic way as before. 
Assume that $A_\phi=((a_{ij}))$ has integer entries.
From $A_\phi A_{T^{\#}}(\sigma)= A_{T}(\sigma)A_\phi$, we see it is symmetric. If $\phi(1,1)=(\frac23, \frac13)$, then we see
\begin{align}
    a_{11}+a_{12}+a_{13}:a_{21}+a_{22}+a_{23}:a_{31}+a_{32}+a_{33}=3:2:1
    \label{3:2:1}.
\end{align}
Assume for now that $\Delta^{\#}[e] \cap\Delta^{\#}[(12)]$ is mapped to $\Delta[e]\cap\Delta[(12)]$. Then from
\begin{align*}
    \phi(1,y)=\left(\frac{a_{21}+a_{22}+a_{23}y}{a_{11}+a_{12}+a_{13}y}, \frac{a_{31}+a_{32}+a_{33}y}{a_{11}+a_{12}+a_{13}y}\right),
\end{align*}
if $\lim_{y\rightarrow \infty}\phi(1,y)=(\frac12,0)$, then  there exists an integer $k$ that we have
\begin{align*}
A_{\phi}=
\left(
\begin{array}{ccc}
* & * & 2k\\
* & * & k\\
2k & k & 0\\
\end{array} 
\right).
\end{align*}
However in this case, it is natural to assume 
$\lim_{x\rightarrow \infty}\phi(x,x)=(0,0)$, and then we have
\begin{align*}
a_{22}+a_{23}=a_{32}=0.
\end{align*}
This implies $k=0$ and clearly we have $\phi(\Delta^{\#}(\sigma))\not=\Delta(\sigma)$ which does not fit our purpose. After this wrong trial, we reach the correct assumption 
that $\Delta^{\#}[e]\cap\Delta^{\#}[(12)]$ is mapped to $\Delta[e]\cap\Delta[(23)]$. 
Indeed if
$\lim_{y\rightarrow \infty}\phi(1,y)=(0,0)$, then 
$A_{\phi}$ has the form
\begin{align*}
A_{\phi}=
\left(
\begin{array}{ccc}
* & * & *\\
* & * & 0\\
* & 0 & 0\\
\end{array} 
\right).
\end{align*}
From $\lim_{x\rightarrow \infty}\phi(x,x)=(\frac12,0)$, we obtain
\begin{align*}
\frac{a_{22}}{a_{12}+a_{13}}=\frac12.
\end{align*}
Considering (\ref{3:2:1}), by several trials we found
\begin{align*}
A_{\phi}=
\left(
\begin{array}{ccc}
1 & 1 & 1\\
1 & 1 & 0\\
1 & 0 & 0\\
\end{array} 
\right)
\end{align*}
which satisfies all the conditions on $\Delta[e]\cup\Delta[(13)]\cup\Delta[(123)]\cup\Delta[(132)]$.
We define an involution on $\mathcal{S}_{n}$.

\begin{figure}[h]
\begin{center}
\begin{tabular}{cc}
\scalebox{1}[1]{
\begin{tikzpicture}[domain=-0.2:5.5]
\draw[very thin,color=gray, dashed] (0,5) -- (5,5);
\draw[thin] (5,0) -- (5,5);
\draw[->] (-0.2,0) -- (5.2,0) ;
\draw[->] (0,-0.2) -- (0,5.2) ;
\draw [domain=0:5]plot(\x,\x);
\draw [domain=5/2:5]plot(\x,{-\x+5)});
\draw [domain=0:5]plot(\x,{\x/2}) ;
\draw [domain=5/2:5]plot(\x,{2*\x-5});
\fill (2.1,0.8) node[below] {e};
\fill (3.5,0.9) node[below] {(12)};
\fill (2.3,1.9) node[below] {(23)};
\fill (3.3,2.9) node[below] {(123)};
\fill (4.3,3) node[below] {(13)};
\fill (4.3,1.8) node[below] {(132)};
\draw[->, ultra thick] (10/3,5/3) -- (0,0) ;
\end{tikzpicture}}
\scalebox{1}[1]{
\begin{tikzpicture}[domain=-0.2:5.5]
\draw[->] (-0.2,0) -- (5.2,0) ;
\draw[->] (0,-0.2) -- (0,5.2) ;
\draw [domain=0:5]plot(\x,\x);
\draw [domain=0:5]plot(\x,5/3);
\draw [domain=0:5]plot(5/3,\x) ;
\fill (2.7,3.9) node[below] {e};
\fill (0.8,3.2) node[below] {(12)};
\fill (4,3) node[below] {(23)};
\fill (0.5,1.4) node[below] {(123)};
\fill (1.1,0.8) node[below] {(13)};
\fill (3,1) node[below] {(132)};
\draw[->, ultra thick] (5/3,5/3) -- (5/3,5) ;
\end{tikzpicture}}
\end{tabular}
\end{center}
\caption{The $1$-time partition of $(\Delta, T_{P})$ and $(\Delta^{\#}, T^{\#}_{P})$.}
\end{figure}
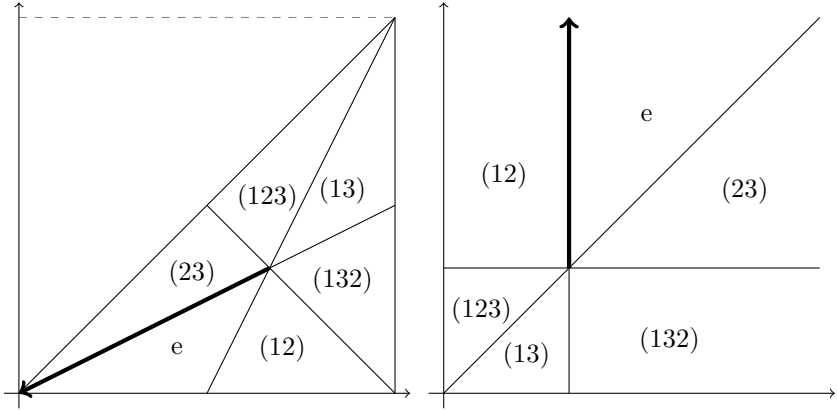

\begin{Def}
We denote the set of involutions of the symmetric group
by
\begin{align*}
    Inv(\mathcal{S}_{n})=\{\sigma\in \mathcal{S}_{n} : \sigma^2=e\}.
\end{align*}
\end{Def}

The cardinality of this set $\#Inv(\mathcal{S}_n)$:$1,2,4,10,26,76,\cdots$ are also known as telephone numbers and various studies have been made on these numbers (see Section 5.1.4 of \cite{{Knuth-TAoCP}}).

\begin{Thm}
The $n$-dimensional Poincar\'{e} algorithm $(\Delta, T)$ is algebraic self-dual on $w_0Inv(\mathcal{S}_{n+1})$ where
\begin{align*}
w_{0}=
\left(
\begin{array}{cccc}
1 & 2 & \cdots & n\\
n & n-1 &\cdots & 1\\
\end{array} 
\right).
\end{align*}
And
\begin{align*}
A_{\phi}=
\left(
\begin{array}{ccc}
1 & \cdots & 1\\
\vdots & \iddots & ~\\
1 & ~ & ~\\
\end{array} 
\right).
\end{align*}
\end{Thm}

\begin{proof}
Let $M=((a_{i,j}))$ be a monomial $(0,1)-$matrix with $a_{i,j}=1$. i.e., there is exactly one ``1'' in each row and each column. We denote
\begin{align*}
M
\longleftrightarrow
\left(
\begin{array}{ccc}
 \cdots & i & \cdots \\
 \cdots & j &\cdots \\
\end{array} 
\right).
\end{align*}
Then, we have
\begin{align*}
w_0\sigma^{-1}&=
\left(
\begin{array}{cccc}
\sigma(1) & \sigma(2) & \cdots & \sigma(n+1)\\
n+1 & n &\cdots & 1\\
\end{array} 
\right)\longleftrightarrow
(A_{\sigma^{-1}(i)j})B, \\
\sigma w_0&=\left(
\begin{array}{cccc}
1 & 2 &\cdots & n+1\\
\sigma(n+1) & \sigma(n) & \cdots & \sigma(1)
\end{array} 
\right)\longleftrightarrow
B(A_{\sigma^{-1}(i)j})^{t}.
\end{align*}

Therefore, we have
\begin{align*}
\{\sigma\in \mathcal{S}_{n+1} : B(A_{\sigma^{-1}(i)j})^{t}=(A_{\sigma^{-1}(i)j})B\}
&=\{\sigma\in \mathcal{S}_{n+1} : (w_0\sigma)^2=e\}\\
&=w_0Inv(\mathcal{S}_{n+1}).
\end{align*}

Let $\phi(B_1,B_2,\dots,B_n)=(b_1,b_2,\dots,b_n)$ and $\sigma\in w_0Inv(\mathcal{S}_{n+1})$. By the definition of $Inv(\mathcal{S}_{n+1})$, 
\begin{align*}
    \sigma^{-1}(i)&=w_{0}\sigma w_{0}(i)\\
    &=w_{0}\sigma(n+1-i+1)\\
    &=n+1-\sigma(n+1-i+1)+1
\end{align*}
and we have 
\begin{align}
    n-\sigma(n-i+2)+1=\sigma^{-1}(i)-1.
\label{Invol-1}
\end{align}

We show $\phi(\Delta^{\#}(\sigma))\supset \Delta(\sigma)$.
Let $(b_1,b_2,\dots,b_n)\in \Delta(\sigma)$.

For $i\not=n-j+1, n-j+2$, by (\ref{Invol-1}), 
\begin{align*}
B_{\sigma(i+1)-1}-B_{\sigma(i)-1}&=\frac{b_{n-\sigma(i+1)+1}-b_{n-\sigma(i+1)+2}}{b_n}-\frac{b_{n-\sigma(i)+1}-b_{n-\sigma(i)+2}}{b_n}\\
&=\frac{b_{\sigma^{-1}(n-i+1)-1}-b_{\sigma^{-1}(n-i+1)}-b_{\sigma^{-1}(n-i+2)-1}+b_{\sigma^{-1}(n-i+2)}}{b_n}\\&>0.
\end{align*}

For $i\not=n-j+1, n-j+2$, since $\sigma(n+1)=j$ and $\sigma\in w_0Inv(\mathcal{S}_{n+1})$, by (\ref{Invol-1}), $\sigma(n-j+2)=1$. Then we have 
\begin{align*}
B_{\sigma(n-j+2)-1}-B_{\sigma(n-j+1)-1}&=1-\frac{b_{n-\sigma(n-j+1)+1}-b_{n-\sigma(n-j+1)+2}}{b_n}\\
&=\frac{b_n-b_{\sigma^{-1}(j+1)-1}+b_{\sigma^{-1}(j+1)}}{b_n}\\
&>0
\end{align*}
and
\begin{align*}
B_{\sigma(n-j+3)-1}-B_{\sigma(n-j+2)-1}&=\frac{b_{n-\sigma(n-j+3)+1}-b_{n-\sigma(n-j+3)+2}}{b_n}-1\\
&=\frac{b_{\sigma^{-1}(j-1)-1}-b_{\sigma^{-1}(j-1)}-b_n}{b_n}\\
&>0.
\end{align*}

Similarly, we have
$\phi(\Delta^{\#}(\sigma))\subset \Delta(\sigma)$.
\end{proof}

\begin{Cor}
The n-dimentional Poincar\'{e} algorithm is symmetric in measure on $w_0Inv(\mathcal{S}_{n+1})$, i.e., for all $\sigma_1, \sigma_2, \dots, \sigma_s\in w_0Inv(\mathcal{S}_{n+1})$,
\begin{align*}
    \mu(\Delta[\sigma_1, \sigma_2, \dots, \sigma_s])=\mu(\Delta[\sigma_s, \sigma_{s-1}, \dots, \sigma_1]).
\end{align*}
\end{Cor}

\newpage
\appendix
\section{The slow version of Garrity-Schweiger map}
F. Schweiger defined the Flip-flop map in \cite{Schweiger-BmS}. 
It is known that the jump transformation of this map is Garrity's triangle map (See also \cite{CAS-slowrriangle}).
Similarly, we consider the $n$-dimensional Flip-flop map, and we can see that the jump map of the map is the Garrity-Schweiger map.

Let $\Delta=\{x \in \mathbb{R}_{>}^n : 1\ge x_1\ge\cdots\ge x_n\}$. Let the cylinder set of the Selmer algorithm and Brun algorithm be $\Delta_S(i)$ and $\Delta_B(i)$ respectively.
Then, since
\begin{align*}
\Delta_S(i)=\{x\in\Delta : x_{i}> 1-x_n \ge x_{i+1}\}, \quad \Delta_B(i)=\{x\in\Delta : x_i> 1-x_1\ge x_{i+1}\},
\end{align*}
we have
\begin{align*}
\Delta=\Delta_{S}(0)\cup\Delta_{B}(n).
\end{align*}

Now, we define the map $T:\Delta\rightarrow\Delta$ as
\begin{align*}
A_{T}=
\left\{
\begin{array}{ll}
  \left(
\begin{array}{cccc}
1 & & & -1\\
  & 1 & & \\
  & & \ddots & \\
  & &  & 1\\
\end{array} 
\right) \quad \textbf{on}~ \Delta_S(0),\\ 
   \left(
\begin{array}{cccc}
  & 1 & & \\
  & & \ddots  & \\
  & & & 1 \\
 1 & -1 &  & \\
\end{array} 
\right) \quad \textbf{on}~ \Delta_B(n).
\end{array} \right. 
\end{align*}

We consider the jump transformation over the cylinder $\Delta_S(0)$, then we obtain a map with matrices
\begin{align*}
 \left(
\begin{array}{cccc}
  & 1 & & \\
  & & \ddots  & \\
  & & & 1 \\
 1 & -1 &  & \\
\end{array} 
\right) 
  \left(
\begin{array}{cccc}
1 & & & -1\\
  & 1 & & \\
  & & \ddots & \\
  & &  & 1\\
\end{array} 
\right) ^k
=
 \left(
\begin{array}{cccc}
  & 1 & & \\
  & & \ddots  & \\
  & & & 1 \\
 1 & -1 &  &-k \\
\end{array} 
\right) 
\end{align*}

This map is Garrity-Schweiger map $T_G$.

\begin{Prop}
The $n$-dimensional Flip-Flop algorithm is algebraic self-dual. And
\begin{align*}
A_{\phi}=
\left(
\begin{array}{ccc}
1 & \cdots & 1\\
\vdots & \iddots & ~\\
1 & ~ & ~\\
\end{array} 
\right).
\end{align*}
\end{Prop}

\end{document}